\documentclass[12pt]{amsart}
\usepackage[matrix,arrow]{xy}
\usepackage{amssymb}
\usepackage{amsmath}
\usepackage{amsfonts}
\usepackage{epsfig}
\usepackage{color}
\usepackage{epstopdf}
\usepackage{graphicx}
\usepackage{amsthm}
\usepackage{enumerate}
\usepackage[mathscr]{eucal}
\usepackage{verbatim}

\usepackage[bookmarks=true,hyperindex,pdftex,colorlinks,citecolor=blue,
linkcolor=blue,urlcolor=blue]{hyperref}
\usepackage{tikz}
\usetikzlibrary{matrix,arrows.meta}

\theoremstyle{plain}

\newtheorem{thm}{Theorem}[section]

\newtheorem{prop}[thm]{Proposition}

\theoremstyle{definition}

\newtheorem{defn}[thm]{Definition}
\newtheorem{rem}[thm]{Remark}

\theoremstyle{plain}

\DeclareMathOperator{\Id}{Id}

\newcommand{\free}{\mathfrak{F}}

\newcommand{\Lin}{\mathcal{L}}
\newcommand{\lip}{_{\mathrm{Lip}}}
\newcommand{\Lip}{{\mathrm{Lip}}_0}

\newcommand{\LipA}{\mathrm{Lip}_{\mathcal{A}}}

\title[The Lipschitz bounded approximation property for operator ideals]
{The Lipschitz bounded approximation property for operator ideals}

\author[G.~Choi]{Geunsu Choi}
\address[G.~Choi]{Department of Mathematics Education, Dongguk University, Seoul 04620, Republic of Korea}
\email{\texttt{chlrmstn90@gmail.com}}

\author[M.~Jung]{Mingu Jung}
\address[M.~Jung]{{Basic Science Research Institute and Department of Mathematics, POSTECH, Pohang 790-784, Republic of Korea} \newline
\href{http://orcid.org/0000-0000-0000-0000}{ORCID: \texttt{0000-0003-2240-2855} }}
\email{\texttt{jmingoo@postech.ac.kr}}

\thanks{The first author was supported by Basic Science Research Program through the National Research Foundation of
Korea(NRF) funded by the Ministry of Education, Science and Technology [NRF-2020R1A2C1A01010377]. The second author was supported by NRF (NRF-2019R1A2C1003857) and by POSTECH Basic Science Research Institute Grant, whose NRF grant number is 2021R1A6A1A10042944} 


\begin{document}

\begin{abstract}

In this article, we introduce the Lipschitz bounded approximation property for operator ideals. With this notion, we extend the original work of Godefroy and Kalton and give some partial answers on equivalence between the bounded approximation property and the Lipschitz bounded approximation property based on an arbitrary operator ideal. Furthermore, we investigate the three space problem for the preceding bounded approximation properties.

\end{abstract}

\maketitle

\section{Introduction}


%

In 2003, Godefroy and Kalton \cite{GK} gave characterizations of the bounded approximation property of a Banach space $X$ in terms of Lipschitz maps on $X$ with finite-dimensional range. More precisely, they proved that a Banach space $X$ has the $\lambda$-bounded approximation property if and only if its Lipschitz-free space over $X$ has the $\lambda$-bounded approximation property if and only if $X$ has the Lipschitz $\lambda$-bounded approximation property, i.e., 
\begin{equation}\label{eq:LipAP}
\Id_X \in \overline{\{ f \in \Lip (X,X) : \dim f(X) < \infty, \|f\|_{\textup{Lip}} \leq \lambda \}}^{\tau_c},
\end{equation}
where $\tau_c$ is the compact-open topology and 
$(\Lip (X,X), \|\cdot\|_{\textup{Lip}})$ is the Banach space of all Lipschitz maps $f : X \to X$ with $f(0)=0$ 
equipped with the Lipschitz norm $\|\cdot\|\lip$ given by 
\[
\| f \|\lip = \sup \left\{ \frac{f(x)-f(y)}{\|x-y\|} : x, y \in X, x\neq y \right\}. 
\] 
Motivated by this fact, we investigate whether the same type of equivalent statements can be obtained when the bounded approximation property on $X$ is replaced by the $\mathcal{A}$-bounded approximation property for an arbitrary operator ideal $\mathcal{A}$ introduced by Oja in \cite{O}. Recall that in the case when $\mathcal{A}$ is the ideal of finite rank operators, the $\mathcal{A}$-bounded approximation property is the classical bounded approximation property.

On the other hand, Godefroy and Saphar \cite{GS} proved the three space problem for the bounded approximation property of a pair $(X,M)$ where $M$ is locally complemented in $X$. Afterwards, Choi and Kim \cite{CK} obtained an analogue for the bounded compact approximation property under some additional conditions. Our aim is to extend these concepts in the settings of arbitrary operator ideals and Lipschitz operator ideals.


To continue with, it requires some basic knowledge about the Lipschitz-free space over a Banach space and about Lipschitz operator ideals, which will be provided in Section \ref{section:preliminaries}. In Section \ref{section:Lip-AP}, we mainly prove many equivalences among those bounded approximation properties naturally-defined with respect to operator ideals, partially extending previous results given by Godefroy and Kalton \cite{GK}. Moreover, we prove further results about inheritance of the approximation properties from a Lipschitz-free space to its own space provided some additional requirements. It will be provided some supplementary equivalence results concerning the approximation properties of $\textup{Lip}_0(X)$ which can be seen as a continuation of the work of Oja \cite{O}. Finally, we present partial answers to the three space problem for the Lipschitz bounded approximation properties in Section \ref{section:TSP}. 
Throughout the paper, we use the standard notation in the Banach space theory. 

\section{Preliminaries}\label{section:preliminaries} 
We assume $X$ and $Y$ always to be real Banach spaces.  
Let $\delta_x$ be the evaluation functional on $\Lip (X) := \textup{Lip}_0 (X,\mathbb{R}) $ with $\delta_x(f)=f(x)$.
The \textit{Lipschitz-free space} over $X$ is defined by
$$\mathfrak{F} (X):=\overline{\text{span}}\{\delta_x\}_{x \in X} \subseteq \Lip (X)^*.$$
Notice that $\delta_X : X \rightarrow \mathfrak{F} (X)$ given by $\delta_X(x)=\delta_x$ is a non-linear isometry (see, for instance, \cite[Proposition 2.1]{GK}). 
It is well-known that the Lipschitz-free space $\mathfrak{F} (X)$ has the following universal property: For any Banach space $Y$ and any $f \in \Lip (X, Y)$, there exists a unique $L_f \in \mathcal L(\mathfrak{F} (X), Y)$ 
such that the following diagram commutes: 
\[
\begin{tikzpicture}
  \matrix (m)
    [
      matrix of math nodes,
      row sep    = 3em,
      column sep = 4em
    ]
    {
      X              & Y\\
      \mathfrak{F}(X) &             \\
    };
  \path
    (m-1-1) edge [->] node [left] {$\delta_X$} (m-2-1)
    (m-1-1.east |- m-1-2)
      edge [->] node [above] {$f$} (m-1-2)
    (m-2-1.east) edge [->]
    node [below] {$\quad L_f$} (m-1-2);
\end{tikzpicture}
\]
and $\|f\|_{\textup{Lip}}=\|L_f\|$ \cite{GK,W}. Moreover, the mapping $f \mapsto L_f$ is an isometric isomorphism from $\Lip (X, Y)$ onto $\mathcal{L} (\mathfrak{F} (X), Y)$. In particular, $\Lip (X)$ is isometrically isomorphic to $\free (X)^*$ and this identification will be used in the sequel without explanation. 
It is also shown in \cite[Lemma 2.2]{GK} that for $f \in \Lip (X, Y)$, there exists a unique $\widehat{f} \in \mathcal L(\mathfrak{F} (X), \mathfrak{F} (Y))$ with $\| \widehat{f} \| = \| f \|\lip$ such that the following diagram commutes: 
\[
\begin{tikzpicture}
  \matrix (m)
    [
      matrix of math nodes,
      row sep    = 3em,
      column sep = 4em
    ]
    {
      X              & Y\\
      \mathfrak{F}(X) &         \mathfrak{F}(Y)    \\
    };
  \path
    (m-1-1) edge [->] node [left] {$\delta_X$} (m-2-1)
    (m-1-1.east) edge [->] node [above] {$f$} (m-1-2)
      (m-1-2) edge [->] node [right] {$\delta_Y$} (m-2-2)
    (m-2-1.east) edge [->]
    node [below] {$\widehat{f}$} (m-2-2);
\end{tikzpicture}
\]
Recall that the map $\beta_X : \mathfrak{F} (X) \rightarrow X$ given by $x^* (\beta_X (\mu)) = \langle x^*, \mu \rangle$ for all $x^* \in X^* \subseteq \Lip (X)$, is a linear quotient map and a left inverse of $\delta_X$. One can easily check that $L_f = \beta_Y \widehat{f}$ for every $f \in \Lip (X, Y)$. 
For more detailed account on Lipschitz-free spaces, we refer to \cite{G, Kalton, Pestov, W}.

Recall the definition of Lipschitz operator ideals. 
\begin{defn}\label{defn:Lip-ideal}
We say that a set $I$ in the space of Lipschitz maps is a \emph{Lipschitz operator ideal} if for every Banach spaces $X$ and $Y$, we have the following:
\begin{enumerate}
\setlength\itemsep{0.4em}
\item[\textup{(i)}] $I(X,Y)$ is a linear subspace of $\textup{Lip}(X,Y)$.
\item[\textup{(ii)}] $h(\cdot)y \in I(X,Y)$ for every $h \in \textup{Lip}(X, \mathbb{R})$ and $y \in Y$.
\item[\textup{(iii)}] For every Banach spaces $W$ and $Z$, $Rfg \in I(W,Z)$ whenever $g \in \textup{Lip}_0(W,X)$, 
$f \in I(X,Y)$ and $R \in \mathcal L(Y,Z)$.
\end{enumerate}
\end{defn}
This definition was introduced by Achour, Rueda, S\'anchez-P\'erez and Yahi \cite{ARSY} and independently by Cabrera-Padilla, Ch\'avez-Dom\'inguez, Jim\'enez-Vargas and Villegas-Vallecillos \cite{CCJV}.  

For a (linear) operator ideal $\mathcal A$, we will consider the composition Lipschitz operator ideal 
$$
\textup{Lip}_{\mathcal A}(X,Y):= \{ f \in \textup{Lip}_0 (X,Y) \colon L_f \in  \mathcal A (\mathfrak{F} (X), Y) \}.
$$
For  $f \in \Lip (X, Y)$, let 
$$
\text{slope} (f) := \left\{ \frac{f(x_1)-f(x_2)}{\|x_1-x_2\|} \colon x_1, x_2 \in X, x_1 \neq x_2 \right\}
\subseteq Y.
$$
It is observed in \cite[Propositions 2.1, 2.2 and 2.4]{JSV} that 
\begin{enumerate}
\setlength\itemsep{0.4em}
\item 
$\textup{Lip}_{\mathcal F}(X,Y)= \{ f \in \textup{Lip}(X,Y) \colon \dim f(X) < \infty \}$,
\item 
$\textup{Lip}_{\mathcal{K}}(X,Y)= \{ f \in \textup{Lip}(X,Y) \colon 
\text{slope} (f) \text{ is relatively compact} \}$,
\item
$\textup{Lip}_{\mathcal{W}}(X,Y) =\{ f \in \textup{Lip}(X,Y) \colon 
\text{slope} (f) \text{ is relatively weakly compact} \}$,
\end{enumerate} 
where $\mathcal{F}, \mathcal{K}$ and $\mathcal W$ denote the operator ideal of all finite rank operators, compact operators and all weakly compact operators, respectively.

For an operator ideal $\mathcal{A}$, Banach spaces $X, Y$ and $T \in \mathcal{A} (X, Y)$, it is clear that 
$$L_T \delta_X=T =T \beta_X \delta_X;$$
hence $L_T = T \beta_X$ on $\mathfrak{F} (X)$ by linearity and continuity of $T$. This implies that $\mathcal{A} \subseteq \LipA$. 




Recall that a Banach space $X$ is said to have the \textit{approximation property} (AP) if for every compact subset $K$ of $X$ and every $\varepsilon > 0$, there exists a finite rank operator $S$ on $X$ such that  $\sup_{x \in K} \|Sx-x\|\leq \varepsilon$, that is, the identity $\Id_X$ of $X$ belongs to $\overline{ \mathcal{F} (X, X) }^{\tau_c}$, where $\tau_c$ denotes the compact-open topology on $X$. For $\lambda \geq 1$, we say that $X$ has the $\lambda$-\textit{bounded approximation property} ($\lambda$-BAP) if $\Id_X$ belongs to $\overline{ \{ T \in \mathcal{F} (X, X) : \|T \| \leq \lambda \}}^{\tau_c}$. We simply say that $X$ has the \textit{bounded approximation property} (BAP) if it has the $\lambda$-BAP for some $\lambda \geq 1$. We refer the reader to \cite{CASAZZA, LT} for a detailed account on the approximation properties.

In what follows, we introduce the definitions of the variants of the bounded approximation properties in which we are mainly interested.  
For simplicity, given $\lambda \geq 1$ and an operator ideal $\mathcal{A}$, let $\mathcal A^{\lambda}(X, Y):=\{T \in \mathcal A(X,Y) : \|T\|\leq \lambda \}$ and $\textup{Lip}_{\mathcal A}^{\lambda}(X,Y):=\{f \in \textup{Lip}_{\mathcal A}(X,Y) : \|f\|_{\textup{Lip}} \leq \lambda \}.$ By the \emph{conjugate operator} $T^*: Y^* \to X^*$ corresponding to $\mathcal{A}$ we mean an operator $T^* \in \mathcal{L}(Y^*,X^*)$ such that $T \in \mathcal{A}(X,Y)$, and we write $\mathcal{A}_c^\lambda(Y^*,X^*) := \{T^* \in \mathcal{L}(Y^*,X^*) : T \in \mathcal{A}^\lambda(X,Y)\}$. 
Following \cite{O}, a Banach space $X$ has the \emph{$\mathcal A$-$\lambda$-bounded approximation property} ($\mathcal{A}$-$\lambda$-BAP) if $\Id_X \in \overline{\mathcal A^{\lambda}(X,X)}^{\tau_c}$. We simply say that $X$ has the \emph{$\mathcal{A}$-BAP} if it has the $\mathcal{A}$-$\lambda$-BAP for some $\lambda \geq 1$. Also, $X^*$ is said to have the \emph{$\mathcal{A}$-$\lambda$-BAP with conjugate operators} if $\Id_{X^*} \in \overline{\mathcal A_c^{\lambda}(X^*,X^*)}^{\tau_c}$, or simply the \emph{$\mathcal{A}$-BAP with conjugate operators} if it holds for some $\lambda \geq 1$.


\begin{defn}
Let $\mathcal A$ be an operator ideal and $\lambda \geq 1$. 
A Banach space $X$ is said to have the \emph{Lipschitz $\mathcal A$-$\lambda$-bounded approximation property} (\textup{Lip}-$\mathcal A$-$\lambda$-BAP) if
$$
\Id_X \in \overline{\textup{Lip}_{\mathcal A}^{\lambda}(X,X)}^{\tau_c}.
$$
Let us simply say that $X$ has the \emph{\textup{Lip}-$\mathcal A$-BAP} if it has the \textup{Lip}-$\mathcal A$-$\lambda$-BAP for some $\lambda \geq 1$. 
\end{defn}

\section{Main results on the $\textup{Lip}$-$\mathcal{A}$-BAP}\label{section:Lip-AP}

In this section, we concentrate on the implicative relations and characterizations on the kinds of newly defined approximation properties. To do so, we first introduce a useful characterization result by arguing that the \textup{Lip}-$\mathcal A$-BAP of a Banach space $X$ is closely related to the $\tau_c$-approximability of the operator $\beta_X : \mathfrak{F} (X) \rightarrow X$. More precisely, it turns out that $\beta_X$ is approximated by Lipschitz maps with respect to $\tau_{c}$ if and only if it is approximable by bounded linear operators with respect to the same topology. Let us denote by $\iota_X$ the canonical embedding from $X^*$ into $\textup{Lip}_0 (X)$. Notice that $\iota_X = \beta_X^*$ up to isometry. 


\begin{thm}\label{thm:X-Lip-BAP}
Let $\mathcal A$ be an operator ideal. The following statements are equivalent for a Banach space $X$.
\begin{enumerate}
\setlength\itemsep{0.4em}
\item[\textup{(a)}] $X$ has the \textup{Lip}-$\mathcal A$-$\lambda$-$BAP$.
\item[\textup{(b)}] $\beta_X \in \overline{\textup{Lip}_{\mathcal A}^\lambda (\mathfrak{F} (X), X ) }^{\tau_{c}}$.
\item[\textup{(c)}] $\beta_X \in \overline{\mathcal A^\lambda (\mathfrak{F} (X), X ) }^{\tau_c}$.
\item[\textup{(d)}] $\iota_X \in \overline{\mathcal{A}_c^\lambda(X^*,\textup{Lip}_0(X))}^{\,w^*}$.
\end{enumerate} 
\end{thm}

In item (d), the notation $w^*$ is used for the weak-star topology on $\mathcal{L} (X^*, \Lip (X))$ with respect to $X^* \widehat{\otimes}_{\pi} \free (X)$, where $X^* \widehat{\otimes}_{\pi} \free (X)$ denotes the projective tensor product of $X^*$ and $\free (X)$.

\begin{proof}[Proof of Theorem \ref{thm:X-Lip-BAP}] 
(c) $\Rightarrow$ (b) follows from the fact that $\mathcal{A} \subseteq \textup{Lip}_{\mathcal A}$. 

(b) $\Rightarrow$ (a).
 Let $K$ be a compact subset of $X$ and let $\varepsilon > 0$
be given. As $\delta_X (K)$ is a compact subset of $\mathfrak{F} (X)$, 
there exists $f \in \textup{Lip}_{\mathcal A}^\lambda (\mathfrak{F}(X), X)$ such that
$$\varepsilon\geq\sup_{x \in K}\|(\beta_X-f)(\delta_x)\|
=\sup_{x \in K}\|\beta_X \delta_X(x)-f \delta_X(x)\|
=\sup_{x \in K}\|x-f\delta_X(x)\|.$$
Since $f \delta_X \in \textup{Lip}_{\mathcal A}^\lambda(X, X)$,
(a) follows. 

(a) $\Rightarrow$ (c).
As any element $\mu \in \mathfrak{F}(X)$ can be approximated by a finite combination of elements in $\delta(X)$, it suffices to claim that for a compact set $K$ in $X$ and $\varepsilon >0$, there exists $L_f \in \mathcal{A}^\lambda (\mathfrak{F} (X), X)$ such that $\sup_{x \in K} \|L_f \delta_X(x) - x \| \leq \varepsilon$. 
Our assumption implies that there exists $f \in \textup{Lip}_{\mathcal A}^\lambda(X, X)$ such that
$$\varepsilon\geq\sup_{x \in K}\|x-f(x)\|
=\sup_{x \in K}\| x -L_f\delta_X(x)\|
=\sup_{x \in K}\|(\beta_X-L_f)(\delta_x)\|.$$
Since $L_f \in \mathcal A^\lambda(\mathfrak{F}(X), X)$, we complete the proof. 

(c) $\Rightarrow$ (d). 
Let $(T_\alpha)$ be a net in $\mathcal{A}^\lambda (\free (X), X)$ such that $(T_\alpha)$ converges to $\beta_X$ in $\tau_c$. For each $\mu \in \free (X)$ and $x^* \in X^*$, we have that 
\[
\langle \mu, T_\alpha^* (x^*)\rangle = x^* (T_\alpha (\mu)) \rightarrow x^* (\beta_X (\mu)) = \langle \mu, \iota_X (x^*) \rangle.
\]
From this, we can deduce that $T_\alpha^*$ converges to $\iota_X$ in $\mathcal{L}(X^*, \Lip (X))=\mathcal{L}(X^*, \free (X)^*) = (X^* \widehat{\otimes}_{\pi} \free (X))^*$ in the weak-star topology since $(T_\alpha)$ is bounded. 

(d) $\Rightarrow$ (c). If we take a net $(T_\alpha)$ in $\mathcal{A}^\lambda (\free (X), X)$ so that $(T_\alpha^*)$ converges to $\iota_X$ in the weak-star topology, then as above, we have that 
\[
x^* (T_\alpha (\mu)) \rightarrow x^* (\beta_X (\mu))
\]
for each $\mu \in \free (X)$ and $x^* \in X^*$. It follows that $(T_\alpha)$ converges to $\beta_X$ in the weak operator topology in $\mathcal{L} (\mathfrak{F}(X), X)$. By considering their convex combinations, we see that $\beta_X$ belongs to the closure in the strong operator topology of $\mathcal{A}^\lambda (\free (X), X)$ which coincides with the $\tau_c$-closure of $\mathcal{A}^\lambda (\free (X), X)$. 

\end{proof}

Next, we show that the \textup{Lip}-$\mathcal{A}$-BAP and the $\mathcal{A}$-BAP are equivalent on a Lipschitz-free space $\mathfrak{F} (X)$ over a Banach space $X$, and observe that the \textup{Lip}-$\mathcal{A}$-BAP of $\mathfrak{F} (X)$ is inherited to $X$. If we focus on the last assertion on the following theorem, one may notice that it is a weaker condition of $\textup{Lip}_0(X)$ having the $\mathcal A$-$\lambda$-BAP with conjugate operators, which will be covered in the paragraph preceding Proposition \ref{prop:Lip(X)-F(X)-BAP}.


\begin{thm}\label{thm:equiv3}
Let $\mathcal A$ be an operator ideal.
 The following statements are equivalent for a Banach space $X$.
 \begin{enumerate}
 \setlength\itemsep{0.4em}
\item[\textup{(a)}] $\mathfrak{F}(X)$ has the \textup{Lip}-$\mathcal A$-$\lambda$-$BAP$.
\item[\textup{(b)}] $\mathfrak{F}(X)$ has the $\mathcal A$-$\lambda$-BAP.
\item[\textup{(c)}] $\delta_X \in \overline{\textup{Lip}_{\mathcal A}^\lambda (X, \mathfrak{F}(X))}^{\tau_c}$.
\item[\textup{(d)}] $\Id_{\Lip (X)} \in \overline{\mathcal{A}_c^\lambda(\textup{Lip}_0(X),\textup{Lip}_0(X))}^{\, w^*}$,
\end{enumerate} 
Furthermore, if one of \textup{(a)}-\textup{(d)} holds, then $X$ has the \textup{Lip}-$\mathcal A$-$\lambda$-$BAP$.
\end{thm}

\begin{proof}
(b) $\Rightarrow$ (a) follows from the fact that $\mathcal{A} \subseteq \textup{Lip}_{\mathcal A}$. 

(a) $\Rightarrow$ (c).
 Let $K$ be a compact subset of $X$ and let $\varepsilon > 0$
be given. 
By (a),
there exists an $f \in \textup{Lip}_{\mathcal A}^{\lambda}(\mathfrak{F}(X), \mathfrak{F}(X))$ such that
$$\sup_{x \in K}\|\delta_X(x)-f \delta_X(x)\|_{\mathfrak{F}(X)}\leq \varepsilon.$$
Since $f \delta_X \in \textup{Lip}_{\mathcal A}^{\lambda}(X, \mathfrak{F}(X))$,
(c) follows. 

(c) $\Rightarrow$ (b). 
As in the proof of (a) $\Rightarrow (c)$ in Theorem \ref{thm:X-Lip-BAP}, it suffices to consider a set $\delta_X (K)$ for a compact subset $K$ of $X$ instead of taking an arbitrary compact subset of $\mathfrak{F} (X)$. 
Let $\varepsilon > 0$ be given. By our assumption, there exists $f \in \textup{Lip}_{\mathcal A}^{\lambda}(X, \mathfrak{F}(X))$ such that
$$\varepsilon\geq\sup_{x \in K}\|\delta_X(x)-f(x)\|_{\mathfrak{F}(X)}
=\sup_{x \in K}\|\delta_X (x) -L_f \delta_X (x)\|_{\mathfrak{F}(X)}.$$
Since $L_f \in \mathcal A^{\lambda}(\mathfrak{F}(X), \mathfrak{F}(X))$, we can conclude that $\mathfrak{F}(X)$ has the $\mathcal A$-$\lambda$-BAP.

(b) $\Leftrightarrow$ (d) can be proved in an analogous way as in Theorem \ref{thm:X-Lip-BAP}. 

Finally, suppose that $\delta_X$ belongs to the $\tau_c$-closure of ${\textup{Lip}^{\lambda}_{\mathcal A}(X, \mathfrak{F}(X))}$. Let $K$ be a compact subset of $X$ and let $\varepsilon > 0$ be given. 
Then there exists $f \in \textup{Lip}^{\lambda}_{\mathcal A}(X, \mathfrak{F}(X))$ such that 
 $$\varepsilon \geq  \sup_{x \in K} \|f(x) - \delta_X(x)\|_{\mathfrak{F}(X)}
 \geq \sup_{x \in K}\|\beta_X f(x) -\beta_X \delta_X(x)\|
 =\sup_{x \in K}\|\beta_X f(x) -x\|.$$
Since $\beta_X f \in \textup{Lip}^{\lambda}_{\mathcal A}(X, X)$, we conclude that $X$ has the \textup{Lip}-$\mathcal A$-$\lambda$-$BAP$.

\end{proof}

It is worthwhile to note here that the equivalence between the $\mathcal{A}$-BAP and the \textup{Lip}-$\mathcal{A}$-BAP for an arbitrary space is unknown contrary to the result of the case $\mathcal{A}=\mathcal{F}$, as the original equivalence in \cite[Theorem 5.3]{GK} heavily depends on the fact that $\mathfrak{F}(E)$ has the 1-BAP when $E$ is finite-dimensional \cite[Proposition 5.1]{GK}. Nevertheless, we are still able to observe that the $\mathcal{A}$-BAP and the \textup{Lip}-$\mathcal{A}$-BAP are equivalent in some particular cases. To begin with, we recall that a Banach space $X$ is said to have the \emph{isometric lifting property} if there exists $U \in \mathcal{L}(X,\mathfrak{F}(X))$ with $\|U\|=1$ such that $\beta_X U = \Id_X$. It is shown that every separable Banach space has the isometric lifting property \cite[Theorem 3.1]{GK}.

\begin{prop}
Let $X$ be a Banach space with the isometric lifting property. Then, $X$ has the $\mathcal{A}$-$\lambda$-BAP if $X$ has the \textup{Lip}-$\mathcal{A}$-$\lambda$-BAP.
\end{prop}

\begin{proof}
As $X$ has the \textup{Lip}-$\mathcal{A}$-$\lambda$-BAP, we have $\beta_X \in \overline{\mathcal{A}^\lambda(\mathfrak{F}(X),X)}^{\tau_c}$ by Theorem \ref{thm:X-Lip-BAP}. Take $(T_\alpha) \subset \mathcal{A}^\lambda(\mathfrak{F}(X),X)$ so that $T_\alpha$ converges to $\beta_X$ in the $\tau_c$-topology. Let us denote by $U$ a norm-one linear operator in $\mathcal{L}(X,\mathfrak{F}(X))$ such that $\beta_X U =\Id_X$. Note that $T_\alpha U$ converges to $\beta_X U = \Id_X$ in the $\tau_c$-topology and that $T_\alpha U \in \mathcal{A}^\lambda(X,X)$, as desired.
\end{proof}

As already shown in Theorem \ref{thm:equiv3}, a Banach space $X$ has the \textup{Lip}-$\mathcal{A}$-$\lambda$-BAP whenever $\mathfrak{F}(X)$ has the property. The following result shows that the inheritance of $\mathcal{A}$-$\lambda$-BAP of the Lipschitz free space $\mathfrak{F}(X)$ to $X$ is also true under some natural assumption on the given operator ideal $\mathcal{A}$.

\begin{prop}\label{prop:equiv4}
Suppose that $\mathcal A$ is an operator ideal and $X$ is a Banach space satisfying that 
 $T^{**} \in \mathcal A(\mathfrak{F}(X)^{**}, \mathfrak{F}(X))$ whenever 
$T \in \mathcal A(\mathfrak{F}(X), \mathfrak{F}(X))$.
Then $X$ has the $\mathcal A$-$\lambda$-BAP if $\mathfrak{F}(X)$ has the $\mathcal A$-$\lambda$-BAP.
\end{prop}

\begin{proof}
We follow the argument used in the proof of \cite[Theorem 5.3]{GK}. Let $\varepsilon > 0$ and a compact set $K \subset X$ be given. Let $\{x_1, \dots, x_n\}$ be an $\varepsilon_0$-net of $K$ and define $E := \operatorname{span}\{ x_1, \dots, x_n\}$, where $\varepsilon_0 = (2+\lambda)^{-1} \varepsilon$. By \cite[Proposition 4.7]{GK}, we can find a linear operator $W : X \rightarrow \mathfrak{F}(X)^{**}$ with $\| W \| =1$ so that $\beta_X^{**} W x = x$ for every $x \in X$ and $W(E) \subset \mathfrak{F} (X)$. Choose $T \in \mathcal{A}^\lambda (\mathfrak{F} (X), \mathfrak{F}(X))$ such that
$$
\| TW(x_j) - W x_j \| < \varepsilon_0
$$
for every $1 \leq j \leq n$. Let us define the map $S := \beta_X^{**} T^{**} W$. Since $T \in \mathcal{A}(\mathfrak{F}(X),\mathfrak{F}(X))$, we have again that $T^{**} \in \mathcal{A}(\mathfrak{F}(X)^{**},\mathfrak{F}(X))$ and hence $S \in \mathcal{A}(X,X)$. Clearly, $\| S \| \leq \lambda$ and 
$$
\| S x_j - x_j \| = \| \beta_X^{**} T^{**} W (x_j) - \beta_X^{**} W (x_j) \| \leq \| TW (x_j) - W x_j \| < \varepsilon_0 
$$
for every $1 \leq j \leq n$. It follows that 
$\| S x - x \|  < \varepsilon $ for every $x \in K$. Thus, $X$ has the $\mathcal{A}$-$\lambda$-BAP.
\end{proof}

In particular, if an operator ideal $\mathcal{A}$ satisfies that $T^{**} \in \mathcal{A}(Z^{**},Z)$ whenever $T \in \mathcal{A}(Z,Z)$ for every Banach space $Z$, then the assertion follows. Thus we have the following direct result due to Schauder's theorem and Gantmacher's theorem.

\ \

\textbf{Example.}
The assumption in Proposition \ref{prop:equiv4} holds for every Banach space $X$ when $\mathcal{A} =\mathcal{F}, \mathcal{K}$ or $\mathcal{W}$.

\ \

Recall that Sinha and Karn introduced a notion of $p$-compact sets and $p$-compact operators in \cite{SK}, and the concept was extended to the Lipschitz case in \cite{ADT}. It is known from \cite[Theorem 4.2]{SK} that the set $\mathcal{K}_p (X, Y)$ of $p$-compact operators becomes an operator ideal with a specific norm equipped on it. According to \cite[Corollary 3.6]{DPS}, an operator $T \in \mathcal{L} (X,Y)$ is $p$-compact if and only if $T^{**} \in \mathcal{L} (X^{**},Y^{**})$ is $p$-compact. Moreover, the operator ideal $\mathcal{K}_p$ is regular, that is, $T \in \mathcal{K}_p (X, Y)$ whenever $j_Y T \in \mathcal{K}_p (X, Y^{**})$, where $j_Y : Y \rightarrow Y^{**}$ is the canonical embedding \cite[Theorem 5]{P}. We also remark that it is clear by definition that a $p$-compact operator is a compact operator. 

On the other hand, it is well-known that the set of $p$-summing operators $\Pi_p (X,Y)$ (for its definition, see \cite{DJT} for instance) is an operator ideal under a specific norm and every $p$-summing operator is weakly compact \cite[Theorem 2.17]{DJT}. Moreover, an operator $T \in \mathcal{L} (X, Y)$ is $p$-summing if and only if $T^{**} \in \mathcal{L} (X^{**},Y^{**})$ is $p$-summing \cite[Proposition 2.19]{DJT}. It is clear by definition that the operator ideal $\Pi_p$ is regular.  
Hence, we have just obtained the following consequence:

\ \

\textbf{Example.}
The assumption in Proposition \ref{prop:equiv4} holds for every Banach space $X$ when $\mathcal{A}=\mathcal{K}_p$ or $\Pi_p$ for $1 \leq p < \infty$.

\ \

Next, we would like to discuss the $\mathcal{A}$-BAP on the space $\textup{Lip}_0(X)$, and observe that this is distinguished from $X$ having the \textup{Lip}-$\mathcal{A}$-BAP. Recall from \cite{O} that a Banach space $X$ is \emph{strongly extendably locally reflexive} if there is $\lambda \geq 1$ such that for all finite-dimensional subspaces $E \subseteq X^{**}$ and $F \subseteq X^*$  and given $\varepsilon>0$, there exists $T \in \mathcal{L}(X^{**},X^{**})$ with $\|T\| \leq \lambda + \varepsilon$ such that
$$
T(E) \subseteq X, \quad T^* (X^*) \subseteq X^* \quad \text{and} \quad (Tx^{**})(x^*)=x^{**}(x^*) \quad \text{for } x^{**} \in E \text{ and } x^* \in F.
$$
One can notice that a consequence from Theorem \ref{thm:equiv3} is that if $\textup{Lip}_0(X)$ has the $\mathcal{A}$-$\lambda$-BAP with conjugate operators, then $\mathfrak{F}(X)$ has the $\mathcal{A}$-$\lambda$-BAP. Let us remark that this can be also derived from a previously known result \cite[Theorem 2.1]{O}.  
A careful examination shows that some of assumptions of \cite[Lemma 3.3]{O} and \cite[Theorem 3.7.(b)]{O} can be formally weakened and yields the following result.
Let us write $\mathcal{A}^{dd} (X, Y) = \{ T \in \mathcal{L} (X, Y) : T^{**} \in \mathcal{A} (X^{**}, Y^{**})\}$.

\begin{prop}\label{prop:Lip(X)-F(X)-BAP}
Let $\mathcal{A}$ be an operator ideal such that $\mathcal{A}=\mathcal{A}^{dd}$. Suppose $\mathcal{A}$ satisfies that $T^{**} \in \mathcal{A}(Z^{**},Z)$ whenever $T \in \mathcal{A}(Z,Z)$ for every Banach space $Z$. Then, the following are equivalent.
\begin{enumerate}
\item[\textup{(a)}] $\textup{Lip}_0(X)$ has the $\mathcal{A}$-BAP with conjugate operators.
\item[\textup{(b)}] $\mathfrak{F}(X)$ has the $\mathcal{A}$-BAP and it is strongly extendably locally reflexive.
\end{enumerate}
\end{prop}

\begin{proof}
(a) $\Rightarrow$ (b). From the comment above, it remains to show that $\mathfrak{F}(X)$ is strongly extendably locally reflexive. Indeed, this follows from the same argument in \cite[Lemma 3.3]{O} as the assumption $\mathcal{A} \subseteq \mathcal{W}$ given there is only used to guarantee $T^{**} (Z^{**}) \subseteq Z$ whenever $T \in \mathcal{A}(Z,Z)$. Consequently, $T^{***} (z^{***}) \in Z^{***}$ is a weak-star continuous linear functional on $Z^{**}$; hence $T^{***} (z^{***})$ is actually an element of $Z^*$ for every $z^{***} \in Z^{***}$. In other words, $T^{***} (Z^{***}) \subseteq Z^*$. 

(b) $\Rightarrow$ (a). This is a consequence of the proof in \cite[Theorem 3.7.(b)]{O}, as the hypothesis $T \in \mathcal{A}(Z,Z)$ implies $T^{**} \in \mathcal{A}(Z^{**},Z)$ may replace the roles of $\mathcal{A} \subseteq \mathcal{W}$ and the regularity of $\mathcal{A}$.
\end{proof}

If we restrict it to the case $\mathcal{A}=\mathcal{F}$, the result of Johnson \cite{J} tells that $X^*$ has the $\lambda$-BAP if and only if $X^*$ has the $\lambda$-BAP with conjugate operators. Having this in mind and using the fact that $X^*$ is one-complemented in $\textup{Lip}_0(X)$ \cite{L}, we obtain the following remark.

\begin{rem}
Let $X$ be a Banach space.
\begin{enumerate}
\item[\textup{(a)}] $\textup{Lip}_0(X)$ has the BAP.
\item[\textup{(b)}] $\mathfrak{F}(X)$ has the BAP and it is strongly extendably locally reflexive.
\item[\textup{(c)}] $X^*$ has the BAP.
\end{enumerate}
Then, we have \textup{(a)} $\Leftrightarrow$ \textup{(b)} $\Rightarrow$ \textup{(c)}.
\end{rem}

We finish the section with a yet another characterization of the $\textup{Lip}$-$\mathcal{A}$-BAP on $X$ and $\mathfrak{F}(X)$ which extends the idea of Grothendieck's characterization of the AP in terms of universal domain and range spaces.

\begin{prop}
Let $\mathcal A$ be an operator ideal. Then, the following statements are equivalent for a Banach space $X$.
\begin{enumerate}
\item[\textup{(a)}] $X$ has the $\textup{Lip}$-$\mathcal{A}$-$\lambda$-BAP.
\item[\textup{(b)}] $\mathcal{L}^1(X,Y) \subset \overline{\textup{Lip}_\mathcal{A}^\lambda(X,Y)}^{\tau_c}$ for every Banach space $Y$.
\item[\textup{(c)}] $\textup{Lip}_0^1(Y,X) \subset \overline{\textup{Lip}_\mathcal{A}^\lambda(Y,X)}^{\tau_c}$ for every Banach space $Y$.
\end{enumerate}
Moreover, the following statements are equivalent, and any of the following items implies the preceding ones.
\begin{enumerate}
\item[\textup{(d)}] $\mathfrak{F}(X)$ has the $\textup{Lip}$-$\mathcal{A}$-$\lambda$-BAP.
\item[\textup{(e)}] $\textup{Lip}_0^1(X,Y) \subset \overline{\textup{Lip}_\mathcal{A}^\lambda(X,Y)}^{\tau_c}$ for every Banach space $Y$.
\end{enumerate}
\end{prop}

\begin{proof}
(b) $\Rightarrow$ (a) and (c) $\Rightarrow$ (a) follow directly from the particular case when $Y=X$.

(a) $\Rightarrow$ (b). Given $\varepsilon>0$, a compact set $K \subset X$ and a Banach space $Y$, let $T \in \mathcal{L}^1(X,Y)$. As $X$ has the $\textup{Lip}$-$\mathcal{A}$-$\lambda$-BAP, there exists $f \in \textup{Lip}_\mathcal{A}^\lambda(X,X)$ such that $\sup_{x \in K} \|f(x) - x\|<\varepsilon$. Then, $\sup_{x \in K} \|Tf(x)-Tx\|<\varepsilon$ and we have that $Tf \in \textup{Lip}_\mathcal{A}^\lambda(X,Y)$.

(a) $\Rightarrow$ (c). Given $\varepsilon>0$, a compact set $K \subset Y$ and a Banach space $Y$, let $f \in \textup{Lip}_0^1(Y,X)$. Since $f(K)$ is compact, there is $g \in \textup{Lip}_\mathcal{A}^\lambda(X,X)$ such that $\sup_{x \in K} \|g(f(x))-f(x)\|<\varepsilon$. The conclusion follows from $g f \in \textup{Lip}_\mathcal{A}^\lambda(Y,X)$.

(e) $\Rightarrow$ (d) is clear since $\delta_X \in \textup{Lip}_0^1(X,\mathfrak{F}(X)) \subset \overline{\textup{Lip}_\mathcal{A}^\lambda(X,\mathfrak{F}(X))}^{\tau_c}$, see Theorem \ref{thm:equiv3}.

(d) $\Rightarrow$ (e). Given $\varepsilon>0$, a compact set $K \subset X$ and a Banach space $Y$, let $f \in \textup{Lip}_0^1(X,Y)$. By Theorem \ref{thm:equiv3}, there exists $g \in \textup{Lip}_\mathcal{A}^\lambda(X,\mathfrak{F}(X))$ such that $\sup_{x \in K} \|g(x)-\delta_X(x)\|<\varepsilon$. Thus the operator $L_f \in \mathcal{L}^1(\mathfrak{F}(X),Y)$ corresponding to $f$ satisfies that $\sup_{x \in K} \|L_f g(x) - L_f \delta_X(x)\| < \varepsilon$, and note that $L_f g \in \textup{Lip}_\mathcal{A}^\lambda (X,\mathfrak{F}(X))$.

Finally, the last statement of Theorem \ref{thm:equiv3} gives the rest implication.
\end{proof}

\begin{rem}
It is not known that whether an analogue of (a) $\Leftrightarrow$ (c) also holds for (d) in terms of universal domain spaces. Using a different language, it is unknown whether (b) and (e) are equivalent for an arbitrary operator ideal $\mathcal{A}$.
\end{rem}

\section{The three-space problem for the $\textup{Lip}$-$\mathcal{A}$-BAP}\label{section:TSP}

In this section, we would like to consider the three-space problem for the $\textup{Lip}$-$\mathcal{A}$-BAP. First, we start with the case when a subspace is complemented.  

\begin{prop}
Let $X$ be a Banach space and $M$ be a complemented subspace of $X$. If $X$ has the $\textup{Lip}$-$\mathcal{A}$-BAP, then so do $M$ and $X/M$.
\end{prop}

\begin{proof}
Suppose that $X$ has the $\textup{Lip}$-$\mathcal{A}$-$\lambda$-BAP for some $\lambda \geq 1$. Let $(f_\alpha) \subset \textup{Lip}_\mathcal{A}^\lambda (X,X)$ such that $f_\alpha$ converges to $\Id_X$ in the $\tau_c$-topology. Let $K \subset M$ be a compact set. Let $P: X \to M$ and $\iota: M \to X$ be canonical projection and inclusion, respectively. Consider $Pf_\alpha \iota: M \to M$, which belongs to $\textup{Lip}_\mathcal{A}^{\lambda'}(M,M)$ for some $\lambda ' \geq 1$. Take $\alpha_0$ such that $\sup_{x \in K} \|(f_\alpha - \Id_X)(x)\|<\dfrac{\varepsilon}{\|P\|}$ whenever $\alpha \geq \alpha_0$. Thus,
$$
\sup_{x\in K} \|(Pf_\alpha\iota - id_M)(x)\| \leq \|P\| \sup_{x \in K} \|(f_\alpha - \Id_X)(x)\|<\varepsilon
$$
for any $\alpha \geq \alpha_0$; so $M$ has the $\textup{Lip}$-$\mathcal{A}$-BAP.

For $X/M$, note that $M$ is the kernel of a projection $Q := \Id_X - \iota P$. Note that $Q(X)$ is a complemented subspace of $X$; hence $Q(X)$ has the $\textup{Lip}$-$\mathcal{A}$-BAP by the above argument. It follows that $X/M$ has the $\textup{Lip}$-$\mathcal{A}$-BAP as it is isomorphic to $Q(X)$.
\end{proof}

On the other hand, the following can be observed easily.

\begin{prop}\label{prop:sum}
Let $X$ and $Y$ be Banach spaces. If $X$ has the \textup{Lip}-$\mathcal{A}$-BAP and $Y$ has the \textup{Lip}-$\mathcal{A}$-BAP, then $X \oplus_\infty Y$ has the \textup{Lip}-$\mathcal{A}$-BAP.
\end{prop}

As $X$ is isomorphic to $M \oplus_\infty X/M$ for a complemented subspace $M \subseteq X$, we are able to obtain the converse of the first stability result thanks to Proposition \ref{prop:sum}.

\begin{prop}
Let $X$ be a Banach space and $M$ be a complemented subspace of $X$. If $M$ and $X/M$ both have the \textup{Lip}-$\mathcal{A}$-BAP, then so does $X$.
\end{prop}

Now, we move on to the case when a subspace is locally complemented in the whole space. A subspace $M$ of a Banach space $X$ is \emph{locally complemented} (or $M$ is an \emph{ideal}) in $X$ if $M^\perp$ is the kernel of a norm-one projection on $X^*$, or equivalently, there exists a Hahn-Banach extension operator $\sigma : M^* \rightarrow X^*$ such that $\sigma (y^*) (y) = y^* (y)$ and $\| \sigma(y^*) \| = \|y^*\|$ for every $y^* \in M^*$ and $y \in M$. 

\begin{thm}\label{thm:ideal_1}
Let $X$ be a Banach space and $M$ be an ideal of $X$. Let $\mathcal{A}$ be an operator ideal such that $\mathcal{A}=\mathcal{A}^{d}$. If $X$ has the \textup{Lip}-$\mathcal{A}$-BAP, then 
\[
\beta_M^* \in \overline{\mathcal{A}^\lambda ({M}^*, {\mathfrak{F}(M)}^*)}^{w^*}
\]
for some $\lambda \geq 1$.
\end{thm}

\begin{proof}
By Theorem \ref{thm:X-Lip-BAP}, we know that $\beta_X \in \overline{ \mathcal{A}^\lambda (\mathfrak{F} (X), X) }^{\tau_c}$ for some $\lambda \geq 1$. Take a net $(T_\alpha)$ in $\mathcal{A}^\lambda (\mathfrak{F}(X),X)$ so that $(T_\alpha)$ converges to $\beta_X$ in the $\tau_c$-topology. Denoting by $\sigma$ the Hahn-Banach extension operator from $M^*$ to $X^*$, consider $\iota^* T_\alpha^* \sigma : M^* \rightarrow \mathfrak{F}(M)^*$, where $\iota : \mathfrak{F}(M) \hookrightarrow \mathfrak{F}(X)$ is the inclusion map. It is clear that $\iota^* T_\alpha^* \sigma$ belongs to $\mathcal{A}^\lambda ({M}^*, {\mathfrak{F}(M)}^*)$. 
Note that 
\begin{align*}
(\iota^* T_{\alpha}^* \sigma)(m^*)(\delta_M (m)) &= \sigma(m^*) (T_\alpha (\delta_X (m)) ) \\
&\rightarrow \sigma(m^*) (\beta_X (\delta_X (m))) = \sigma(m^*) (m) = m^* (m) 
\end{align*}
for all $m \in M$ and $m^* \in M^*$. As $\mathfrak{F}(M) = \overline{\operatorname{span}}\{\delta_x \colon x \in M\}$ and $\| \iota^* T_{\alpha}^* \sigma \| \leq \lambda$, we can deduce that the net $(\iota^* T_{\alpha}^* \sigma)$ satisfies that 
\[
\langle m^* \otimes \mu,  \iota^* T_{\alpha}^* \sigma \rangle \rightarrow \langle m^* \otimes \mu, \beta_M^* \rangle
\]
for all $m^* \in M^*$ and $\mu \in \frak{F}(M)$. Thus, a standard density argument shows that $(\iota^* T_{\alpha}^* \sigma)$ converges to $\beta_M^*$ in the weak-star topology. 
\end{proof}

Recall from \cite{CK} that the dual space $X^*$ is said to have the bounded weak-star density for compact operators (for short, B$\text{W}^*$D) if $\mathcal{K}^1 (X^*, X^*) \subseteq \overline{ \mathcal{K}_{w^*}^\lambda (X^*, X^*) }^{w^*}$ for some $\lambda \geq 1$, where $\mathcal{K}_{w^*}^\lambda (X^*, X^*)$ is the space of compact operators which are weak-star to weak-star continuous on $X^*$ with norm at least $\lambda$. It is known \cite[Proposition 2.7]{CK} that if $X^*$ is reflexive or has the BAP, then $X^*$ has the B$\text{W}^*$D while the converse is false. In the same paper, it is also observed \cite[Theorem 4.2]{CK} that for an ideal $M$ in a Banach space $X$, if $X$ has the $\mathcal{K}$-BAP and $M^*$ has the B$\text{W}^*$D, then $M$ also has the $\mathcal{K}$-BAP. 
The following shows that the same result can be obtained for an ideal $M$ with the isometric lifting property in $X$ when the $\mathcal{K}$-BAP assumption on $X$ is replaced by the \textup{Lip}-$\mathcal{K}$-BAP.

For Banach spaces $X$ and $Y$, recall that $(\mathcal{L} (X, Y), \tau_c)^*$ consists of all functionals $f$ of the form $f(T) = \sum_n y_n^* (Tx_n)$, where $(x_n) \subset X$, $(y_n^*) \subset Y^*$ and $\sum_n \|x_n\|\|y_n^*\| < \infty$.

\begin{prop}\label{prop:BWD}
Let $X$ be a Banach space and $M$ be an ideal in $X$ with the isometric lifting property. If $X$ has the \textup{Lip}-$\mathcal{K}$-BAP and $M^*$ has the B$\text{W}^*$D, then $M$ has the $\mathcal{K}$-BAP.  
\end{prop}

\begin{proof}
Thanks to Theorem \ref{thm:ideal_1}, we have that $\beta_M^* \in \overline{\mathcal{K}^\lambda (M^*,{ \mathfrak{F}(M)}^*)}^{w^*}$ for some $\lambda \geq 1$. Since $M$ has the isometric lifting property, there exists $S \in \Lin (M, \mathfrak{F}(M))$ such that $\|S\|=1$ and $\beta_M S = \Id_M$. Observe that 
\[
\Id_M^* = (\beta_M S)^* = S^* \beta_M^* \in \overline{ \mathcal{K}^\lambda (M^*, M^* ) }^{w^*}. 
\]
As $M^*$ has the B$\text{W}^*$D, we have that $\Id_M^* \in \overline{ \mathcal{K}_{w^*}^{\lambda'} (M^*, M^* ) }^{w^*}$ for some $\lambda' \geq 1$. Using the weak-star to weak-star continuity, we can assert that there exists a net $(T_\alpha)$ in $\mathcal{K}^{\lambda'} (M, M)$ such that $\phi (T_\alpha)$ converges to $\phi(\Id_M)$ for each $\phi \in (\Lin (M, M), \tau_c)^*$. By a convex combination argument, we conclude that $\Id_M$ belongs to $\overline{ \mathcal{K}^{\lambda'} (M, M) }^{\tau_c}$.
\end{proof}

On the contrary, we prove in the following theorem that $\beta_X^*$ lies in the weak-star closure of $\mathcal{A}^\lambda(X^*,\mathfrak{F}(X)^*)$ provided stronger assumptions on the ideal and again this is a bit weaker than $X$ having the $\textup{Lip}$-$\mathcal{A}$-BAP due to Theorem \ref{thm:X-Lip-BAP}. We will see later in Proposition \ref{prop:K-BWD} that the $\textup{Lip}$-$\mathcal{A}$-BAP may also be derived in some specific case.

\begin{thm}\label{thm:ideal_2}
Let $X$ be a Banach space and $M$ be a subspace of $X$ such that $\mathfrak{F}(M)$ is an ideal in $\mathfrak{F}(X)$. Suppose that $\mathcal{A}$ is an operator ideal such that $\mathcal{A}=\mathcal{A}^{d}$. If $M$ has the \textup{Lip}-$\mathcal{A}$-BAP and $\mathfrak{F}(X)/\mathfrak{F}(M)$ has the $\mathcal{A}$-BAP, then
\[
\beta_X^* \in \overline{\mathcal{A}^\lambda(X^*,{\mathfrak{F}(X)}^*)}^{w^*}
\]
for some $\lambda \geq 1$.  
\end{thm}

\begin{proof}
Take $(g_\alpha) \subset \textup{Lip}_\mathcal{A}^{\lambda_1} (M,M)$ so that $g_\alpha$ converges to $\Id_M$ in the $\tau_c$-topology for some $\lambda_1 \geq 1$. Let us consider $S_\alpha := \sigma g_\alpha^t \iota^* : X^* \to \textup{Lip}_0(X)$ where $g_\alpha^t \in \mathcal{L}(M^*,\textup{Lip}_0(M))$ is the Lipschitz transpose of $g_\alpha$ for each $\alpha$ (see \cite{JSV}), $\iota : M \to X$ is the canonical inclusion and $\sigma : \mathfrak{F}(M)^* \to \mathfrak{F}(X)^*$ the Hahn-Banach extension operator. Then 
\[
S_\alpha \in \mathcal{A}^{\lambda_1} (X^*,\mathfrak{F}(X)^*) = \mathcal{A}^{\lambda_1} (X^*,\textup{Lip}_0(X)).
\]
Let $S$ be a weak-star accumulation point of $(S_\alpha)$ in $\mathcal{L}(X^*,\textup{Lip}_0(X))$. Note that 
\begin{align*}
(S_\alpha x^*)(m) = \sigma(g_\alpha^t(\iota^* x^*))(m) = \sigma((\iota^*x^*g_\alpha))(m) = (\iota^*x^*g_\alpha)(m) \longrightarrow x^*(m)
\end{align*}
for every $m \in M$ and $x^* \in X^*$. Define $R(x^*) := Sx^*-x^*$ for every $x^* \in X^*$. Then $R$ maps from $X^*$ to $\mathfrak{F}(M)^\perp$, where it is known \cite[Lemma 2.27]{W} that 
$$
\mathfrak{F}(M)^\perp = \{ f \in \textup{Lip}_0(X) \colon f(m)=0 \text{ for every } m \in M\}.
$$
Let us denote by $j : \mathfrak{F}(M)^\perp \to \textup{Lip}_0(X)$ the canonical embedding. Note that $S-jR \in \mathcal{L}(X^*,\textup{Lip}_0(X))$. As a matter of fact,
$$
Sx^*-jRx^* = Sx^*-j(Sx^*-x^*)=x^*
$$
for every $x^* \in X^*$. It follows that $S-jR = \beta_X^*$. Recall that $S$ belongs to the weak-star closure of $\mathcal{A}^\lambda(X^*,\textup{Lip}_0(X))$.

We claim that $jR$ belongs to the weak-star closure of $\mathcal{A}^{\lambda_2} (X^*,\textup{Lip}_0(X))$ for some $\lambda_2 \geq 1$. From the assumption that $\mathfrak{F}(X)/\mathfrak{F}(M)$ has the $\mathcal{A}$-BAP, we can take 
\[
(Q_\beta) \subset \mathcal{A}^{\lambda_2} \left( \mathfrak{F}(X)/\mathfrak{F}(M), \mathfrak{F}(X) / \mathfrak{F}(M) \right)
\]
so that $Q_\beta$ converges to $\Id_{X/M}$ in the $\tau_c$-topology for some $\lambda_2 \geq 1$. Then $Q_\beta^* \in \mathcal{A}(\mathfrak{F}(M)^\perp,\mathfrak{F}(M)^\perp)$ and $jQ_\beta^*R \in \mathcal{A}^{\lambda_2} (X^*,\textup{Lip}_0(X))$. Observe that $jQ_\beta^*R \to jR$ in the weak-star topology in $\mathcal{L}(X^*,\textup{Lip}_0(X))$; hence the claim is established. It follows that $\beta_X^* = S-j R$ belongs to the weak-star closure of $\mathcal{A}^{\lambda_1 + \lambda_2} (X^*,\textup{Lip}_0(X))$. 
\end{proof}

\begin{rem}
It is known \cite[Theorem 2.3]{Sofi} that if a subspace $M$ of a Banach space $X$ satisfies that $\mathfrak{F}(M)$ is an ideal in $\mathfrak{F}(X)$, then $M$ is an ideal in $X$. Very recently, it has been discovered in \cite{Abe} that the converse is also true. 
\end{rem}

In the similar spirit as in Proposition \ref{prop:BWD}, we obtain the following result. 

\begin{prop}\label{prop:K-BWD}
Let $X$ be a Banach space and $M$ be a subspace of $X$ such that $\mathfrak{F}(M)$ is an ideal of $\mathfrak{F}(X)$. If $M$ has the \textup{Lip}-$\mathcal{K}$-BAP, $\mathfrak{F}(X)/\mathfrak{F}(M)$ has the $\mathcal{K}$-BAP and $X^*$ has the B$\text{W}^*$D, then $X$ has the \textup{Lip}-$\mathcal{K}$-BAP.
\end{prop}

\begin{proof}
As $\mathfrak{F} (X) / \mathfrak{F} (M)$ has the $\mathcal{K}$-BAP, arguing as in Theorem \ref{thm:ideal_2}, we can deduce that $\beta_X^*$ belongs to the weak-star closure of $\mathcal{K}^{\lambda'}(X^*,\mathfrak{F}(X)^*)$ for some $\lambda' \geq 1$. 
Since $X^*$ has the B$\text{W}^*$D, we can find a net $(T_\alpha)$ in $\mathcal{K}^{\lambda'}(\mathfrak{F} (X), X)$ so that $T_\alpha^*$ converges to $\beta_X^*$ in the weak-star topology. Now, Theorem \ref{thm:X-Lip-BAP} completes the proof.
\end{proof} 

\proof[Acknowledgements]

The authors would like to thank Yun Sung Choi, Ju Myung Kim and Abraham Rueda Zoca for fruitful conversations on the topic of the paper.
The authors would also like to thank the anonymous referees for helpful comments and for pointing out a mistake in the previous version of this paper.

\end{document}